\documentclass[12pt, reqno]{amsart}
\usepackage{amsmath, amsthm, amscd, amsfonts, amssymb, graphicx, color}
\usepackage[bookmarksnumbered, colorlinks, plainpages]{hyperref}
\hypersetup{colorlinks=true,linkcolor=red, anchorcolor=green, citecolor=cyan, urlcolor=red, filecolor=magenta, pdftoolbar=true}

\newtheorem{theorem}{Theorem}[section]
\newtheorem{lemma}[theorem]{Lemma}

\newtheorem{corollary}[theorem]{Corollary}
\theoremstyle{definition}

\theoremstyle{remark}
\newtheorem{remark}[theorem]{\bf{Remark}}
\numberwithin{equation}{section}
\begin{document}

%-------------------------------------------------------------------------
% editorial commands: to be inserted by the editorial office
%
%\firstpage{1} \volume{228} \Copyrightyear{2004} \DOI{003-0001}
%
%
%\seriesextra{Just an add-on}
%\seriesextraline{This is the Concrete Title of this Book\br H.E. R and S.T.C. W, Eds.}
%
% for journals:
%
%\firstpage{1}
%\issuenumber{1}
%\Volumeandyear{1 (2004)}
%\Copyrightyear{2004}
%\DOI{003-xxxx-y}
%\Signet
%\commby{inhouse}
%\submitted{March 14, 2003}
%\received{March 16, 2000}
%\revised{June 1, 2000}
%\accepted{July 22, 2000}
%
%
%
%---------------------------------------------------------------------------
%Insert here the title, affiliations and abstract:
%

\title[Inequality and equality conditions for numerical radius ]{Development of inequality and characterization of equality conditions for the numerical radius }

%\author[Pintu Bhunia and Kallol Paul]{Pintu Bhunia and Kallol Paul}

\author[P. Bhunia]{Pintu Bhunia} 
\address{
Department of Mathematics,
Jadavpur University,
Kolkata 700032,
West Bengal, India}
\email{pintubhunia5206@gmail.com}

 %\thanks{We would like to thank the referee for his/her helpful suggestions.
%Mr. Pintu Bhunia would like to thank UGC, Govt. of India for the financial support in the form of SRF. Prof. Kallol Paul would like to thank RUSA 2.0,
 % Jadavpur   University for the partial support.}

%----------Author 2    \at
\author[K. Paul]{Kallol Paul}
\address{%
Department of Mathematics,
Jadavpur University,
Kolkata 700032,
West Bengal, India}
\email{kalloldada@gmail.com}
%----------classification, keywords, date


%\date{January 1, 2004}
%----------additions
%\dedicatory{Mr. Pintu Bhunia would like to thank UGC, Govt. of India for the financial support in the form of SRF.}
%%% ----------------------------------------------------------------------

\begin{abstract}
Let $A$ be a bounded linear operator on a complex Hilbert space and $\Re(A)$ ( $\Im(A)$ ) denote the real part (imaginary part)  of A.  Among other refinements of the lower bounds for the numerical radius of $A$, we prove that  
\begin{eqnarray*}
w(A)&\geq &\frac{1}{2} \left \|A \right\| +  \frac{  1}{2} \mid \|\Re(A)\|-\|\Im(A)\|\mid,\,\,\mbox{and}\\
w^2(A)&\geq& \frac{1}{4} \left \|A^*A+AA^* \right\| +  \frac{1}{2}\mid \|\Re(A)\|^2-\|\Im(A)\|^2 \mid,
\end{eqnarray*}
where $w(A)$ is the numerical radius of the operator $A$. 
We study the equality conditions for $w(A)=\frac{1}{2}\sqrt{\|A^*A+AA^*\|}$ 
and  prove that $w(A)=\frac{1}{2}\sqrt{\|A^*A+AA^*\|} $  if and only if the numerical range of $A$ is a circular disk with center at the origin and radius $\frac{1}{2}\sqrt{\|A^*A+AA^*\|} $.
 We also obtain upper bounds for the numerical radius of commutators of operators which improve on the existing ones.

\end{abstract}

%%% ----------------------------------------------------------------------
\maketitle
%%% ----------------------------------------------------------------------
%\tableofcontents
\section{Introduction}
\label{intro}
Let $\mathcal{B}(\mathcal{H})$ denote the $C^*$-algebra of all bounded linear operators on a complex Hilbert space $\mathcal{H}$ with inner product $\langle.,.\rangle$. As usual the norm induced by the inner product $\langle.,.\rangle$  is denoted by $ \Vert . \Vert.$ For $A\in \mathcal{B}(\mathcal{H})$, let $W(A)$ denote the numerical range of $A$, which is defined as $W(A)=\{\langle Ax,x\rangle: x\in \mathcal{H},\|x\|=1\}.$
 For $A\in \mathcal{B}(\mathcal{H})$, let $w(A)$ and $\|A\|$ denote the numerical radius and the operator norm of $A$, respectively, defined as,  $w(A)=\sup\{|\langle Ax,x\rangle|: x\in \mathcal{H},\|x\|=1\}$ and $\|A\|=\sup\{\|Ax\|: x\in \mathcal{H},\|x\|=1\}.$ It is easy to verify that $w(.)$ defines a norm on $\mathcal{B}(\mathcal{H})$, which is equivalent to the operator norm $\|.\|$. In fact, for every $A\in \mathcal{B}(\mathcal{H})$, we have the following inequality
\begin{eqnarray}\label{equivalent norm}
\frac{1}{2}\|A\|\leq w(A)\leq \|A\|.
\end{eqnarray}
The Crawford number of $A$, denoted by $c(A),$  is another important numerical constant associated with the numerical range and  is defined as $c(A)=\inf\{|\langle Ax,x\rangle|: x\in \mathcal{H},\|x\|=1\}.$
The adjoint of an operator $A$ is denoted by $A^*$. Clearly $w(A)=w(A^*)$ and $c(A)=c(A^*)$. 
For   $A \in \mathcal{B}(\mathcal{H}),$   the real part and imaginary part of $A,$ denoted as     ${\Re}(A)$ and ${\Im (A)},$ are defined respectively as,   ${\Re (A)} =\frac{A+A^*}{2}$ and ${\Im}(A)=\frac{A-A^*}{2{\rm i}}.$  Thus, $A= {\Re}(A)+ \rm i {\Im}(A)$ is the Cartesian decomposition of $A$. It is well known that, for $A\in \mathcal{B}(\mathcal{H})$,  $w(A)=\sup_{\theta\in \mathbb{R}}\|\Re(e^{\rm i \theta }A)\|=\sup_{\theta\in \mathbb{R}}\|\Im(e^{\rm i \theta}A)\|$. Over the years various refinements of (\ref{equivalent norm}) have been obtained, we refer the interested readers to see \cite{aok,P8,P14,BP,K05,K03,TY} and the references therein.

In this paper, we obtain lower bounds for the numerical radius of bounded linear operators which refine the well known lower bound $w(A)\geq \frac{\|A\|}{2}$ and the bound $w^2(A)\geq \frac{1}{4}\|A^*A+AA^*\|,$  obtained by Kittaneh \cite[Th. 1]{K05}. We also present the equivalent conditions for  the equality $w(A)= \frac{\|A\|}{2}$ as well as $w^2(A)=\frac{1}{4}\|A^*A+AA^*\|.$ 
Further, applying the lower bounds obtained here, we obtain upper bounds for the numerical radius of commutators of bounded linear operators, which refine the existing ones \cite{FH,HK}.

\section{Main results}\label{sec1}

We begin with by noting  an elementary equality of real numbers, $\max\{a,b\}=\frac{a+b}{2}+\frac{|a-b|}{2}$ for all $a,b\in \mathbb{R}$. By using this maximum function we obtain the following lower bound for the numerical radius of bounded linear operators.

\begin{theorem}\label{th3}
	Let $A \in \mathcal{B}(\mathcal{H})$, then 
	\begin{eqnarray*}
		w(A)&\geq& \frac{ \left \|A \right\|}{2} +  \frac{   \big|~~\|\Re(A)\|-\|\Im(A)\| ~~ \big|}{2}.
	\end{eqnarray*}
\end{theorem}
\begin{proof}
Let $x$ be an unit vector in $\mathcal{H}$. Then
it follows from the Cartesian decomposition of $A$ that $|\langle Ax,x\rangle|^2=|\langle \Re(A)x,x\rangle|^2+|\langle \Im(A)x,x\rangle|^2.$
This implies $w(A) \geq \| \Re(A) \|$ and $w(A) \geq \| \Im(A) \|$.
Thus, we have 
\begin{eqnarray*}
w(A)&\geq& \max\{ \| \Re(A) \|,  \| \Im(A) \|   \}\\
&=& \frac{ \| \Re(A) \|+ \| \Im(A) \|}{2}+ \frac{\big|~~\|\Re(A)\|-\|\Im(A)\| ~~ \big|}{2}\\
&\geq& \frac{ \| \Re(A) +\rm i \Im(A) \|}{2}+ \frac{\big|~~\|\Re(A)\|-\|\Im(A)\| ~~ \big|}{2}\\
&=&  \frac{ \left \|A \right\|}{2} +  \frac{   \big|~~\|\Re(A)\|-\|\Im(A)\| ~~ \big|}{2},
\end{eqnarray*}	
as desired.
	
\end{proof}

\begin{remark}
	(i) We note that if $A$ is Hermitian or skew Hermitian operator then the  inequality in Theorem \ref{th3} becomes an equality.\\
(ii) Clearly, the inequality in Theorem \ref{th3} is  stronger than the first inequality in (\ref{equivalent norm}) when $\|\Re(A)\|\neq\|\Im(A)\|$.
\end{remark}

As a consequence of  Theorem \ref{th3} we prove the following corollary.

\begin{corollary}\label{c1}
	Let $A \in \mathcal{B}(\mathcal{H})$. If $w(A)=\frac{\|A\|}{2}$, then
	$\|\Re(A)\|=\|\Im(A)\|=\frac{\|A\|}{2}.$
\end{corollary}
\begin{proof}
	From Theorem \ref{th3}, we have 
	$	w(A)\geq \frac{ \left \|A \right\|}{2} +  \frac{   \left|~~\|\Re(A)\|-\|\Im(A)\| ~~ \right|}{2}\geq \frac{ \left \|A \right\|}{2}.$
This implies that if $w(A)=\frac{ \left \|A \right\|}{2}$, then $\|\Re(A)\|=\|\Im(A)\|$. Also
\[\|\Re(A)\|\leq w(A)=\frac{ \left \|A \right\|}{2}=\frac{ \left \|\Re(A)+\rm i \Im(A)\right\|}{2}\leq \frac{\|\Re(A)\|+\|\Im(A)\|}{2}    =\| \Re(A)\|\]
and so  $\|\Re(A)\|=\|\Im(A)\|=\frac{\|A\|}{2}.$
	\end{proof}
 
Note that the converse of Corollary \ref{c1} is not true, in general.  We next  concentrate our study on an  equivalent condition for $w(A)=\frac{\|A\|}{2}$.

\begin{theorem}\label{theorem1}
Let $A \in \mathcal{B}(\mathcal{H})$. Then the following are equivalent. \\
(i) $w(A)=\frac{\|A\|}{2}.$ \\
(ii) $\|\Re(e^{\rm i \theta }A)\|=\|\Im(e^{\rm i \theta}A)\|=\frac{\|A\|}{2}$ for all $\theta \in \mathbb{R}.$
\end{theorem}

\begin{proof}
(ii) implies (i) is trivial. We only prove (i) implies (ii). Let  $w(A)=\frac{ \left \|A \right\|}{2}$.
Then from Corollary \ref{c1}, we have
$\|\Re(A)\|=\|\Im(A)\|=\frac{\|A\|}{2}.$ Clearly,  for all $\theta \in \mathbb{R}$, $e^{\rm i \theta}A \in \mathcal{B}(\mathcal{H})$ and $w(e^{\rm i \theta}A)=w(A)$, $\|e^{\rm i \theta}A\|=\|A\|$. Thus, 	$w(A)=\frac{\|A\|}{2}$ implies $\|\Re(e^{\rm i \theta }A)\|=\|\Im(e^{\rm i \theta}A)\|=\frac{\|A\|}{2}$ for all $\theta \in \mathbb{R}.$

\end{proof}

\begin{remark}
Yamazaki \cite[Th. 3.1]{TY} proved that the following are equivalent:\\
		(i) $ w(A)=\frac{\|A\|}{2}. $ \\
		(ii) $ \|\Re(e^{\rm i \theta }A)\|+\|\Im(e^{\rm i \theta}A)\|=\|A\| $ for all $\theta \in \mathbb{R}.$ \\
We note that the Theorem \ref{theorem1} gives an alternate elementary proof of  \cite[Th. 3.1]{TY}.
\end{remark}

We next prove that,  $ \|A\| = \sqrt{\|A^*A+AA^*\|}= \sqrt{\|A^*A-AA^*\|},$  if  $ w(A) = \frac{1}{2} \|A\|.$  To do so we need the following lemma.
\begin{lemma}[\cite{BB}] \label{lemmaA^*B}
Let $A,B\in \mathcal{B}(\mathcal{H})$ be non zero. Then $\|A+B\|=\|A\|+\|B\|$ if and only if $\|A\|\|B\|\in \overline{W(A^*B)}.$ 
\end{lemma}
\begin{theorem}\label{normA}
Let $A \in \mathcal{B}(\mathcal{H})$. If $w(A)=\frac{\|A\|}{2}$, then $$\|A\|^2=\|A^*A+AA^*\|=\|A^*A-AA^*\|.$$
\end{theorem}

\begin{proof} We note that 
 for all $\theta \in \mathbb{R},$ 
 $$\|A\|=\|\Re(e^{\rm i \theta }A)+\rm i \Im(e^{\rm i \theta}A)\|\leq \|\Re(e^{\rm i \theta }A)\|+\|\Im(e^{\rm i \theta}A)\|=\|A\|.$$ 
 Then using Lemma \ref{lemmaA^*B} we get, 
 $\|\Re(e^{\rm i \theta }A)\| \|\Im(e^{\rm i \theta}A)\|\in \overline{W\left (\rm i~~ \Re(e^{\rm i \theta }A)\Im(e^{\rm i \theta }A)\right)}$.
 Clearly, $ \|\Re(e^{\rm i \theta }A)\| \|\Im(e^{\rm i \theta}A)\|\leq w( \rm i~~ \Re(e^{\rm i \theta }A)\Im(e^{\rm i \theta }A) ) \leq \|\rm i~~ \Re(e^{\rm i \theta }A)\Im(e^{\rm i \theta }A) \| \leq \|\Re(e^{\rm i \theta }A)\| \|\Im(e^{\rm i \theta}A)\|. $  Since  $\|\Re(e^{\rm i \theta }A)\| \|\Im(e^{\rm i \theta}A)\|$ $\in \mathbb{R}$, 
 $\|\Re(e^{\rm i \theta }A)\| \|\Im(e^{\rm i \theta}A)\| $ $= \left \| \Re \left ( \rm i~~\Re(e^{\rm i \theta }A)\Im(e^{\rm i \theta }A)\right) \right\|.$   
 Clearly, $ \Re \left ( \rm i~~\Re(e^{\rm i \theta }A)\Im(e^{\rm i \theta }A)\right)=\frac{1}{4}(A^*A-AA^*). $     Thus we get, 
\begin{eqnarray}\label{2}
	\|\Re(e^{\rm i \theta }A)\|\|\Im(e^{\rm i \theta}A)\|=\frac{1}{4}\|A^*A-AA^*\|,~~\textit{for all}~~\theta \in \mathbb{R}.
\end{eqnarray}
 Also,
 \begin{eqnarray}\label{1.2}
 	\|\Re(e^{\rm i \theta }A)\|+\|\Im(e^{\rm i \theta}A)\|=\|A\|,~~\textit{for all}~~\theta \in \mathbb{R}.
 \end{eqnarray}
Therefore, from (\ref{2}) and (\ref{1.2}), we have
$$ \| \Re(e^{\rm i \theta }A)\| =\frac{1}{2} \left( \|A\|\pm \sqrt{\|A\|^2-\|A^*A-AA^*\|}\right) ~~\mbox{for all $\theta \in \mathbb{R}.$}$$ 
From Theorem \ref{theorem1} we get,  $ \| \Re(e^{\rm i \theta }A)\| = \frac{1}{2}\|A\|$  for all $\theta \in \mathbb{R},$  and so, $\|A\|^2-\|A^*A-AA^*\|=0$. Now, $\|A^*A-AA^*\|= \|A\|^2\leq \|A^*A+AA^*\|\leq 4w^2(A)=\|A\|^2$.  Hence, $\|A\|^2=\|A^*A-AA^*\|=\|A^*A+AA^*\|$.
\end{proof}

\begin{remark}
Kittaneh \cite[Prop. 1]{K05} proved that if $A^2=0$ then $\|A\|^2=\|AA^*-A^*A\|=\|AA^*+A^*A\|$, whereas Theorem \ref{normA} says that if $w(A)=\frac{\|A\|}{2}$ then $\|A\|^2=\|AA^*-A^*A\|=\|AA^*+A^*A\|$. Clearly, $\left\{A\in \mathcal{B}(\mathcal{H}):A^2=0 \right\}$ $ \subseteq \left\{A\in \mathcal{B}(\mathcal{H}):w(A)=\frac{\|A\|}{2}\right \} $ is proper.
Thus, Theorem \ref{normA} is applicable to larger class of operators than \cite[Prop. 1]{K05}.
	
\end{remark}

In the next theorem we obtain another lower bound for the numerical radius.

\begin{theorem}\label{corr}
	Let $A \in \mathcal{B}(\mathcal{H})$, then
	\begin{eqnarray*}
		w^2(A)&\geq& \frac{1}{4} \left \|A^*A+AA^* \right\| +  \frac{1}{2}\left|~\|\Re(A)\|^2-\|\Im(A)\|^2~~ \right|.
	\end{eqnarray*}
\end{theorem}

\begin{proof}
Let $x$ be an unit vector in $\mathcal{H}$. Then
it follows from the Cartesian decomposition of $A$ that $|\langle Ax,x\rangle|^2=|\langle \Re(A)x,x\rangle|^2+|\langle \Im(A)x,x\rangle|^2.$
This implies $w(A) \geq \| \Re(A) \|$ and $w(A) \geq \| \Im(A) \|$ and so,
	\begin{eqnarray*}
		w^2(A)&\geq& \max \left\{\| \Re(A) \|^2,  \| \Im(A) \|^2  \right\}\\
		&=& \frac{\| \Re(A) \|^2+  \| \Im(A) \|^2}{2}+ \frac{  \left| \| \Re(A) \|^2- \| \Im(A) \|^2\right|}{2} \\
		&=& \frac{\|(\Re(A))^2 \|+\|(\Im(A))^2\|}{2}+ \frac{  \left| \| \Re(A) \|^2- \| \Im(A) \|^2\right|}{2} \\
		&\geq& \frac{\|(\Re(A))^2 +(\Im(A))^2  \|}{2}+ \frac{  \left| \| \Re(A) \|^2- \| \Im(A) \|^2\right|}{2} \\
		&=& \frac{1}{4} \left \|A^*A+AA^* \right \|+ \frac{  \left| \| \Re(A) \|^2- \| \Im(A) \|^2\right|}{2},
	\end{eqnarray*}
	as required.

\end{proof}

\begin{remark}
	Clearly, Theorem \ref{corr} is a non trivial improvement of $w^2(A)$ $\geq \frac{1}{4}\|A^*A+AA^*\|$,
obtained by Kittaneh \cite[Th. 1]{K05}.
\end{remark}

Now, using Crawford number we obtain our next refinement.

\begin{theorem}\label{th1}
Let $A \in \mathcal{B}(\mathcal{H})$, then
\begin{eqnarray*}
w^2(A)&\geq& \frac{1}{4} \left \|A^*A+AA^* \right\|+\frac{c^2(\Re(A))+c^2(\Im(A))}{2}\\
&& +\left|  \frac{\|\Re(A)\|^2-\|\Im(A)\|^2}{2} + \frac{c^2(\Im(A))- c^2(\Re(A))}{2}   \right|.
\end{eqnarray*}
\end{theorem} 
\begin{proof}
Let $x$ be an unit vector in $\mathcal{H}$. Then
it follows from the Cartesian decomposition of $A$ that 
$|\langle Ax,x\rangle|^2=|\langle \Re(A)x,x\rangle|^2+|\langle \Im(A)x,x\rangle|^2.$
	This implies the following two inequalities: $ w^2(A) \geq \| \Re(A) \|^2+c^2(\Im(A)) $ and $w^2(A) \geq \| \Im(A) \|^2+c^2(\Re(A)).$
Therefore, we have
\begin{eqnarray*}
	w^2(A)&\geq& \max \left\{\| \Re(A) \|^2+c^2(\Im(A)),  \| \Im(A) \|^2+c^2(\Re(A))  \right\}\\
	&=& \frac{\| \Re(A) \|^2+c^2(\Im(A))+  \| \Im(A) \|^2+c^2(\Re(A))}{2}\\
	&&+ \left|\frac{\| \Re(A) \|^2+c^2(\Im(A))- \| \Im(A) \|^2-c^2(\Re(A))}{2} \right|\\
	&=& \frac{\|\Re(A) \|^2+\|\Im(A)\|^2}{2}+\frac{c^2(\Re(A))+c^2(\Im(A))}{2}\\
	&& +\left|  \frac{\|\Re(A)\|^2-\|\Im(A)\|^2}{2} + \frac{c^2(\Im(A))- c^2(\Re(A))}{2}   \right| \\
	&=& \frac{\|(\Re(A))^2 \|+\|(\Im(A))^2\|}{2}+\frac{c^2(\Re(A))+c^2(\Im(A))}{2}\\
	&& +\left|  \frac{\|\Re(A)\|^2-\|\Im(A)\|^2}{2} + \frac{c^2(\Im(A))- c^2(\Re(A))}{2}   \right| \\
	&\geq& \frac{\|(\Re(A))^2 +(\Im(A))^2  \|}{2}+\frac{c^2(\Re(A))+c^2(\Im(A))}{2}\\
	&& +\left|  \frac{\|\Re(A)\|^2-\|\Im(A)\|^2}{2} + \frac{c^2(\Im(A))- c^2(\Re(A))}{2}   \right| \\
	&=& \frac{1}{4} \left \|A^*A+AA^* \right \|+\frac{c^2(\Re(A))+c^2(\Im(A))}{2}\\
	&& +\left|  \frac{\|\Re(A)\|^2-\|\Im(A)\|^2}{2} + \frac{c^2(\Im(A))- c^2(\Re(A))}{2}   \right|,
	\end{eqnarray*}
as required.
	
\end{proof}

  \begin{remark}
   In \cite[Cor. 2.3]{P2}, the authors obtained that if $A\in \mathcal{B}(\mathcal{H})$, then
  \begin{eqnarray}\label{bhuniapaul}
  	w^2(A)&\geq& \frac{1}{4} \left \|A^*A+AA^* \right\|+\frac{1}{2}\left(c^2(\Re(A))+c^2(\Im(A))\right).
  \end{eqnarray}
Clearly, Theorem \ref{th1} is stronger than (\ref{bhuniapaul}).
  \end{remark}

Based on Theorem \ref{th1}  we prove the following corollary.

\begin{corollary}\label{cor-1}
	Let $A\in \mathcal{B}(\mathcal{H})$. If $w(A)=\frac{\sqrt{\left\|A^*A+AA^*\right\|}}{2}$, then the following results hold:
	
	$(i)$ There exist norm one sequences $\{x_n\}$ and $\{y_n\}$ in $\mathcal{H}$ such that 
	\begin{eqnarray*}
	|\langle \Re(A) x_n,x_n\rangle|\to 0 ~~\mbox{and} ~~	|\langle \Im(A) ~y_n,y_n\rangle|\to 0.
	\end{eqnarray*}

$(ii)$ $\| \Re(A)\|=\|\Im(A)\|=\frac{\sqrt{\left\|A^*A+AA^*\right\|}}{2}$.
\end{corollary}

\begin{proof}
It follows from Theorem \ref{th1} that
\begin{eqnarray*}
	w^2(A)&\geq& \frac{1}{4} \left \|A^*A+AA^* \right\|+\frac{c^2(\Re(A))+c^2(\Im(A))}{2}\\
	&& +\left|  \frac{\|\Re(A)\|^2-\|\Im(A)\|^2}{2} + \frac{c^2(\Im(A))- c^2(\Re(A))}{2}   \right|\\
	&\geq& \frac{1}{4} \left \|A^*A+AA^* \right\|.
\end{eqnarray*}
This implies that $c(\Re(A))=c(\Im(A))=0$, and so (i) holds. Also we have, $\| \Re(A)\|=\|\Im(A)\|$. Now the last equality of (ii) follows from the following chain of inequalities:
\[\| \Re(A)\|^2 \leq w^2(A)=\frac{\|A^*A+AA^*\|}{4}=\frac{1}{2} \left\|(\Re(A))^2+(\Im(A))^2   \right\| \leq \| \Re(A)\|^2.\]
\end{proof}

We observe that the reverse part of Corollary \ref{cor-1} may not be true, in general. In this context, we obtain an equivalent condition for $w(A)=\frac{\sqrt{\left\|A^*A+AA^*\right\|}}{2}$.

\begin{theorem}\label{theorem10}
	Let $A\in \mathcal{B}(\mathcal{H})$. Then $w^2(A)=\frac{1}{4}{\left\|A^*A+AA^*\right\|}$ if and only if $\|\Re(e^{\rm i \theta}A)\|^2=\|\Im(e^{\rm i \theta}A)\|^2=\frac{1}{4}\|A^*A+AA^*\|$ for all $\theta \in \mathbb{R}$.
\end{theorem}

\begin{proof}
	The sufficient part is trivial, so we only  prove the necessary part. Let $w^2(A)=\frac{1}{4}\|A^*A+AA^*\|.$ Let $\theta\in \mathbb{R}$ be arbitrary. Then by simple computation we have, $\left(\Re(e^{\rm i \theta}A)\right )^2+\left(\Im(e^{\rm i \theta}A)\right)^2=\frac{A^*A+AA^*}{2}.$ Now, we have 
\begin{eqnarray*}
\frac{1}{4}\|A^*A+AA^*\| &=& \frac{1}{2} \left\|\left(\Re(e^{\rm i \theta}A)\right )^2+\left(\Im(e^{\rm i \theta}A)\right)^2 \right\| \\
&\leq& \frac{1}{2} \left( \left\| \Re(e^{\rm i \theta}A) \right\| ^2+\left\| \Im(e^{\rm i \theta}A)\right\| ^2 \right)\\
&\leq& w^2(A)\\
&=& \frac{1}{4}\|A^*A+AA^*\|.
\end{eqnarray*}
This implies that  $ \left\| \Re(e^{\rm i \theta}A) \right\| ^2+\left\| \Im(e^{\rm i \theta}A)\right\| ^2 = \frac{1}{2}\|A^*A+AA^*\|.$
Now, $$\sup_{\theta\in \mathbb{R}}\left\| \Re(e^{\rm i \theta}A) \right\| ^2=\frac{1}{4}\|A^*A+AA^*\|=\sup_{\theta\in \mathbb{R}}\left\| \Im(e^{\rm i \theta}A) \right\| ^2.$$ Thus, $\|\Re(e^{\rm i \theta}A)\|^2=\|\Im(e^{\rm i \theta}A)\|^2$ $=\frac{1}{4}\|A^*A+AA^*\|$ for all $\theta \in \mathbb{R}$.
\end{proof}

In the next theorem we characterize numerical range of an operator when the numerical radius  attains its lower bounds, namely, $w(A)= \frac{\|A\|}{2} $ and $w(A)=\frac{\sqrt{\|A^*A+AA^*\|}}{2}$.

To do so we need the following lemma. This was known for matrices, (see \cite{MP}). 

\begin{lemma}\label{thapp}
	Let $A\in \mathcal{B}(\mathcal{H})$. Then the following are equivalent:\\
		(i) $\|\Re(e^{\rm i \theta }A)\|=k,$ (k is a constant) for all $\theta \in \mathbb{R}$.\\
	(ii) ${W(A)}$ is a circular disk with center at the origin and radius $k$.
\end{lemma}

\begin{proof}
	$(i)\Rightarrow (ii).$ We have, $w\left(\Re(e^{\rm i \theta }A)\right)=k$ for all $\theta \in \mathbb{R}$. So, $\sup_{\|x\|=1}|\langle \Re(e^{\rm i \theta }A)x,x\rangle |$ $ =k$, i.e., $\sup_{\|x\|=1}|Re (e^{\rm i \theta }\langle Ax,x\rangle )|=k$ for all $\theta \in \mathbb{R}$. Thus, for each $\theta \in \mathbb{R}$, there exist a norm one sequence $\{x_n^{\theta} \}$ in $\mathcal{H}$ such that $|Re (e^{\rm i \theta }\langle Ax_{n}^{\theta},x_{n}^{\theta}\rangle )|\to k$. This implies that the boundary of  $W(A)$ must be a circle with center at the origin and radius $k$. Since $W(A)$ is a convex subset of $\mathbb{C}$, so  ${W(A)}$ is a circular disk with center at the origin and radius $k$.\\
	$(ii)\Rightarrow (i).$ Follows easily.
\end{proof}

The desired theorem now follows easily from  Theorem \ref{theorem1} and Theorem \ref{theorem10}.

\begin{theorem}\label{prop12}
	Let $A\in \mathcal{B}(\mathcal{H})$. Then we have,\\
	(i) $w(A)=\frac{1}{2}\|A\|$ if and only if ${W(A)}$ is a circular disk with center at the origin and radius $\frac{1}{2}\|A\|$.\\
	(ii) $w(A)=\frac{1}{2}\sqrt{\|A^*A+AA^*\|}$ if and only if ${W(A)}$ is a circular disk with center at the origin and radius $\frac{1}{2}\sqrt{\|A^*A+AA^*\|}$.
\end{theorem}

\begin{remark}
	For $A \in \mathcal{B}(\mathcal{H})$, here we note that  $w(A)=\frac{\|A\|}{2}$ implies $w(A)=\frac{1}{2}\sqrt{\|A^*A+AA^*\|}$. However, $w(A)=\frac{1}{2}\sqrt{\|A^*A+AA^*\|}$ does not always imply  $w(A)=\frac{\|A\|}{2}$. Consider  $A=\left(\begin{array}{ccc}
	0&1&0\\
	0&0&1\\
	0&0&0
	\end{array}   \right)$. Then we have, $w(A)=\frac{1}{2}\sqrt{\|A^*A+AA^*\|}=\frac{1}{\sqrt{2}}>\frac{1}{2}= \frac{\|A\|}{2}$.
	
\end{remark}

Our final inequality in this section is as follows.

\begin{theorem}\label{th4}
Let $A \in \mathcal{B}(\mathcal{H})$, then
\begin{eqnarray*}
	w^4(A)&\geq& \frac{1}{16} \left \| (A^*A+AA^*)^2+4\left (\Re(A^2) \right )^2 \right\| +  \frac{1}{2}\left|\|\Re(A)\|^4-\|\Im(A)\|^4\right|  \\
	&\geq& \frac{1}{16}\|A^*A+AA^*\|^2+\frac{1}{4}c\left(\left (\Re(A^2)\right )^2\right)+\frac{1}{2}\left|\|\Re(A)\|^4-\|\Im(A)\|^4 \right|. 
\end{eqnarray*}	
\end{theorem}

\begin{proof}
	It follows from $w^4(A)\geq \max\{\|\Re(A)\|^4,\|\Im(A)\|^4\}$ that 
	\begin{eqnarray*}
		w^4(A)&\geq& \frac{ \| \Re(A) \|^4+ \| \Im(A) \|^4}{2}+ \frac{\big|~~\|\Re(A)\|^4-\|\Im(A)\|^4 ~~ \big|}{2}\\
		&\geq& \frac{ \|( \Re(A))^4 +( \Im(A) )^4\|}{2}+ \frac{\big|~~\|\Re(A)\|^4-\|\Im(A)\|^4 ~~ \big|}{2}\\
		&=&  \frac{1}{16} \left \| (A^*A+AA^*)^2+4\left (\Re(A^2) \right )^2 \right\| +  \frac{\left|~~\|\Re(A)\|^4-\|\Im(A)\|^4~~ \right|}{2}  \\
		&\geq& \frac{1}{16}\|A^*A+AA^*\|^2+\frac{1}{4}c\left(\left (\Re(A^2)\right )^2\right)+\frac{\left|~~\|\Re(A)\|^4-\|\Im(A)\|^4~~ \right|}{2}.
	\end{eqnarray*}
\end{proof}

\begin{remark}
 Bag et al. \cite[Th. 8]{S1} obtained that the following inequality:
\begin{eqnarray}\label{bag5}
	w^4(A)
	&\geq& \frac{1}{16}\|A^*A+AA^*\|^2+\frac{1}{4}c\left(\left (\Re(A^2)\right )^2\right).
\end{eqnarray}
Bhunia et al. \cite[Cor. 2.8]{P1}  improved on the inequality (\ref{bag5}) to proved that 
\begin{eqnarray}\label{bhuniaetal}
	w^4(A)&\geq& \frac{1}{16} \left \| (A^*A+AA^*)^2+4\left (\Re(A^2) \right )^2 \right\|.
\end{eqnarray}
Clearly, Theorem \ref{th4} improves on both the inequalities (\ref{bag5}) and (\ref{bhuniaetal}).
\end{remark}

%...................................................

\section{Application to estimate numerical radius of commutators of operators}

In this section we obtain upper bounds for the numerical radius of commutators of bounded linear operators as an application of the lower bounds in the previous section. First we prove the following theorem.

\begin{theorem}\label{th2}
	Let $A,B,X,Y\in \mathcal{B}(\mathcal{H})$, then 
	\begin{eqnarray*}
		w(AXB \pm BYA) \leq 2\sqrt{2}\|B\|\max  \left\{\|X\|,\|Y\| \right\}\sqrt{ w^2(A)-\nu   },
	\end{eqnarray*}
where $\nu=\frac{c^2(\Re (A))+c^2(\Im (A))}{2}+\left|  \frac{\|\Re(A)\|^2-\|\Im(A)\|^2}{2} + \frac{c^2(\Im(A))- c^2(\Re(A))}{2}   \right|$.
	
\end{theorem}

\begin{proof}
	First we assume that $\|X\|\leq 1$ and $\|Y\|\leq 1$. Let $x$ be an unit vector in ${\mathcal{H}}$. Then we have,
	\begin{eqnarray*}
		|\langle (AX\pm YA)x,x\rangle|&\leq&  |\langle Xx,A^*x\rangle|+|\langle Ax,Y^*x\rangle| \\
		&\leq& \|A^*x\|+ \|Ax\|,~~\|Xx\|\leq 1~~\textit{and}~~ \|Y^*x\|\leq 1 \\ 
		&\leq &  \sqrt{2(\|A^*x\|^2+ \|Ax\|^2)}\\ %,~~\textit{by convexity of $ f(x)=x^2$}\\
		&\leq&  \sqrt{2\|AA^*+A^*A\|}\\
		&\leq& 2\sqrt{2}\sqrt{ w^2(A)-\nu  }, ~~\textit{using Theorem \ref{th1}},
	\end{eqnarray*}
 where $\nu=\frac{c^2(\Re (A))+c^2(\Im (A))}{2}+\left|  \frac{\|\Re(A)\|^2-\|\Im(A)\|^2}{2} + \frac{c^2(\Im(A))- c^2(\Re(A))}{2}   \right|$.
Hence, by taking supremum over $\|x\|=1$ we get,
	\begin{eqnarray}\label{eqnth1}
	w(AX\pm YA)&\leq& 2\sqrt{2}\sqrt{ w^2(A)-\nu }.
	\end{eqnarray}
	Now we consider the general case, i.e., $X,Y\in \mathcal{B}(\mathcal{H})$ be arbitrary operators. If $X=Y=0$ then Theorem \ref{th2} holds trivially. Let $\max  \left\{\|X\|,\|Y\| \right\}\neq 0.$ Then clearly $\left \| \frac{X}{\max  \left\{\|X\|,\|Y\| \right\}}\right\|\leq 1$ and $\left \| \frac{Y}{\max  \left\{\|X\|,\|Y\| \right\}}\right\|\leq 1$. So,  replacing $X$ and $Y$ by $\frac{X}{\max  \left\{\|X\|,\|Y\| \right\}}$ and $\frac{Y}{\max  \left\{\|X\|,\|Y\| \right\}}$, respectively, in (\ref{eqnth1}) we get,
	\begin{eqnarray}\label{eqnth2}
	w(AX\pm YA)\leq 2\sqrt{2}\max  \left\{\|X\|,\|Y\| \right\}\sqrt{ w^2(A)-\nu }.
	\end{eqnarray}
	Now replacing $X$ by $XB$ and $Y$ by $BY$ in (\ref{eqnth2}) we get,
	\begin{eqnarray*}
		w(AXB\pm BYA)\leq 2\sqrt{2}\max  \left\{\|XB\|,\|BY\| \right\}\sqrt{ w^2(A)-\nu }.
	\end{eqnarray*}
	This implies that
	\begin{eqnarray*}
		w(AXB\pm BYA) \leq 2\sqrt{2} \|B\| \max  \left\{\|X\|,\|Y\| \right\}\sqrt{ w^2(A)-\nu },
	\end{eqnarray*} as desired.
\end{proof}

Considering $X=Y=I$ in Theorem \ref{th2}, we get the following inequality.

\begin{corollary}\label{corth2}
	Let $A,B\in \mathcal{B}(\mathcal{H})$, then
	\begin{eqnarray}\label{eqncor1}
w(AB\pm BA) &\leq & 2\sqrt{2} \|B\| \sqrt{ w^2(A)-\nu },
	\end{eqnarray}
	where $\nu=\frac{c^2(\Re (A))+c^2(\Im (A))}{2}+\left|  \frac{\|\Re(A)\|^2-\|\Im(A)\|^2}{2} + \frac{c^2(\Im(A))- c^2(\Re(A))}{2}   \right|$.

\end{corollary}

\begin{remark}
Fong and Holbrook \cite{FH} obtained that the following inequality 
\begin{eqnarray}\label{Fong}
w(AB+ BA) \leq  2\sqrt{2} \|B\| w(A).
\end{eqnarray}
Hirzallah and Kittaneh \cite{HK} improved on the inequality (\ref{Fong}) to prove that
\begin{eqnarray}\label{Hirzallah}
w(AB \pm BA)&\leq& 2\sqrt{2}\|B\|\sqrt{ w^2(A)-\frac{|~~\|\Re (A)\|^2-\|\Im (A)\|^2~~|}{2} }.
\end{eqnarray}	
We would like to remark that the inequality obtained in Corollary \ref{corth2} is a proper refinement of that in  (\ref{Hirzallah}). Consider $A=\left(\begin{array}{cc}
20 & 0 \\
0 & 30+ 30 \,{\rm i}
\end{array}\right)$ and $B=\left(\begin{array}{cc}
1 & 0 \\
0 & -1
\end{array}\right)$. Then   Corollary \ref{corth2} gives $w(AB \pm BA)\leq 105.830052$, whereas (\ref{Hirzallah}) gives $w(AB \pm BA)\leq 120.$

\end{remark}

Proceeding similarly as in Corollary \ref{corth2} and using Theorem \ref{th4}, we get the following inequality.

\begin{corollary}\label{corth3}
	Let $A,B\in \mathcal{B}(\mathcal{H})$, then
	\begin{eqnarray*}
		&&	w(AB\pm BA) \\
		&&\leq  2\sqrt{2} \|B\| \left({ w^4(A)-\frac{1}{4}c\left(\left (\Re(A^2)\right )^2\right)-\frac{\left|~~\|\Re(A)\|^4-\|\Im(A)\|^4~~ \right|}{2} }\right)^{\frac{1}{4}}\\
		& &\leq 2\sqrt{2} \|B\| \left({ w^4(A)-\frac{\left|~~\|\Re(A)\|^4-\|\Im(A)\|^4~~ \right|}{2} }\right)^{\frac{1}{4}}.
	\end{eqnarray*}

\end{corollary}

\begin{remark}
	Clearly, Corollary \ref{corth3} is an  improvement of the inequality (\ref{Fong}), obtained in \cite{FH} .
\end{remark}

% ------------------------------------------------------------------------

\subsection*{Acknowledgments} Mr. Pintu Bhunia would like to thank UGC, Govt. of India for the financial support in the form of SRF.

% ------------------------------------------------------------------------

\begin{thebibliography}{1}
	
	
\bibitem {aok} A. Abu-Omar and F. Kittaneh, A generalization of the numerical radius, Linear Algebra  Appl. 569 (2019) 323-334.

\bibitem{BB} M. Barraa and M. Boumazgour, Inner derivations and norm equality, Proc. Amer.
Math. Soc. 130 (2002) 471-476.	
	
\bibitem{S1} S. Bag, P. Bhunia  and K. Paul, Bounds of numerical radius of bounded linear operator using $t$-Aluthge transform, Math. Inequal. Appl. 23(3) (2020) 991-1004.


\bibitem{P8} P. Bhunia, K. Paul and R.K. Nayak, Sharp inequalities for the numerical radius of Hilbert space operators and operator matrices, Math. Inequal. Appl. 24(1) (2021) 167-183. 	

\bibitem{P14}  P. Bhunia and K. Paul, New upper bounds for the numerical radius of Hilbert space operators, Bull. Sci. math. 167 (2021) 102959. {https://doi.org/10.1016/j.bulsci.2021.102959}

\bibitem{P1} P. Bhunia, S. Bag and K. Paul, Bounds for zeros of a polynomial using numerical radius of Hilbert space operators, Ann. Funct. Anal. 12, 21 (2021). {https://doi.org/10.1007/s43034-020-00107-4}

\bibitem {BP} P. Bhunia and K. Paul, Some improvements of numerical radius inequalities of operators and operator matrices, Linear Multilinear Algebra, (2020). {https://doi.org/10.1080/03081087.2020.1781037}


\bibitem{P2} P. Bhunia and K. Paul, Refinements of norm and numerical radius inequalities, Rocky Mountain J. Math. (2021), In press.


	
	
\bibitem{FH} C.K. Fong and J.A.R. Holbrook, Unitarily invariant operator norms, Canad. J. Math. 35(1983) 274-299.

%\bibitem{F} C.K. Fong, Norm estimates related to self-commutators, Linear Algebra Appl. 74 (1986) 151-156.	
	
\bibitem {GR} K.E. Gustafson and D.K.M. Rao, Numerical Range, The field of values of linear operators and matrices, Springer, New York, 1997.	
	
\bibitem{HK} O. Hirzallah and F. Kittaneh, Numerical radius inequalities for several operators, Math. Scand. 114(1) (2014) 110-119.
	
	
\bibitem {K05} F. Kittaneh, Numerical radius inequalities for Hilbert spaces operators, Studia Math. 168 (2005) 73-80.	

\bibitem {K03} F. Kittaneh, Numerical radius inequality and an estimate for the numerical radius of the Frobenius companion matrix, Studia Math. 158 (2003) 11-17.


\bibitem{MP} M. Marcus and C. Pesce, Computer generated numerical ranges and some resulting theorems, Linear Multilinear Algebra 20 (1987) 121-157.
	
	
\bibitem{TY} T. Yamazaki, On upper and lower bounds of the numerical radius and an equality condition, Studia Math. 178(1) (2007) 83-89.



\end{thebibliography}
\end{document}